\DeclareSymbolFont{AMSb}{U}{msb}{m}{n}
\documentclass[12pt,a4paper,twoside,leqno,noamsfonts]{amsart}
\usepackage[a4paper,top=3cm,bottom=3cm,left=2.8cm,right=2.8cm]{geometry}
\usepackage[utf8]{inputenc}
\usepackage[english]{babel}
\usepackage{amsmath, amsthm, amstext,amsfonts,mathrsfs}
\usepackage{microtype}
\usepackage{indentfirst}
\usepackage[bookmarks]{hyperref}
\usepackage{braket}
\usepackage{caption}
\usepackage{stmaryrd}
\usepackage{comment}
\usepackage{multirow}
\usepackage{booktabs}
\usepackage[all]{xy}
\usepackage{graphicx}
\usepackage{xstring}
\usepackage{enumerate}
\usepackage[shortlabels]{enumitem}

\usepackage{listings}
\usepackage{amstext} 
\usepackage{array}   
\newcolumntype{L}{>{$}l<{$}} 
\newcolumntype{C}{>{$}c<{$}}

\newtheorem{corollary}{Corollary}[section]
\newtheorem*{corollary*}{Corollary}
\newtheorem{lemma}[corollary]{Lemma}
\newtheorem*{lemma*}{Lemma}
\newtheorem{theorem}[corollary]{Theorem}
\newtheorem*{theorem*}{Theorem}

\newtheorem*{conjecture*}{Conjecture}
\newtheorem{proposition}[corollary]{Proposition}
\newtheorem*{proposition*}{Proposition}

\theoremstyle{definition}
\newtheorem{definition}[corollary]{Definition}

\newtheorem{remark}[corollary]{Remark}

\newcommand{\bC}{\mathbb{C}}

\newcommand{\bP}{\mathbb{P}}

\newcommand{\bT}{\mathbb{T}}

\newcommand{\bZ}{\mathbb{Z}}

\newcommand{\cH}{\mathcal{H}}

\newcommand{\cN}{\mathcal{N}}
\newcommand{\cO}{\mathcal{O}}

\newcommand{\cQ}{\mathcal{Q}}

\newcommand{\cT}{\mathcal{T}}

\newcommand{\cV}{\mathcal{V}}

\newcommand{\sF}{\mathscr{F}}

\newcommand{\sO}{\mathscr{O}}

\newcommand{\res}{\operatorname{res}}

\newcommand{\im}{\operatorname{Im}}

\newcommand{\hol}{\operatorname{hol}}

\newcommand{\dd}{d}
\newcommand{\st}{\mbox{ s.t. }}
\newcommand{\supp}[1]{|#1|}

\title{Hyperelliptic odd coverings}

\author{Riccardo Moschetti}
\address[R.M.]{ Dipartimento di Matematica ``F. Casorati'', Universit\`{a} degli Studi di Pavia, Via Ferrata 5, 27100 Pavia, Italy}
\email{riccardo.moschetti@unipv.it}

\author{Gian Pietro Pirola}
\address[G.P.P]{ Dipartimento di Matematica ``F. Casorati'', Universit\`{a} degli Studi di Pavia, Via Ferrata 5, 27100 Pavia, Italy}
\email{gianpietro.pirola@unipv.it}

\begin{document}

\begin{abstract}
We investigate a class of odd (ramification) coverings $C \to \bP^1$ where $C$ is hyperelliptic, its Weierstrass points maps to one fixed point of $\bP^1$ and the covering map makes the hyperelliptic involution of $C$ commute with an involution of $\bP^1$. We show that the total number of hyperelliptic odd coverings of minimal degree $4g$ is ${3g \choose g-1} 2^{2g}$ when $C$ is general.
Our study is approached from three main perspectives: if a fixed effective theta characteristic is fixed they are described as a solution of a certain class of differential equations; then they are studied from the monodromy viewpoint and a deformation argument that leads to the final computation.
\end{abstract}
\maketitle

\section{Introduction}
Let $C$ be a compact Riemann surface. We say that a covering $H: C \to \bP^1$  has odd ramification if all the ramification points of $H$ have odd degree, i.e. $z \mapsto z^{2n+1}$ in suitable local coordinates. The structure and the importance of odd ramification coverings have been highlighted in the seminal works of Serre \cite{MR1076476, MR1078120} and Fried \cite{MR2735035}. 
In particular, the monodromy group of any odd ramification covering is a subgroup of the alternating group $A_d$, where $d$ is the degree of  $H$. Moreover, the covering comes with an associated spin structure i.e. a theta characteristic of $C$, that is a line bundle such that $L^2=\omega_C,$ where $\omega_C$ is the canonical bundle of $C$. 
The standard construction of odd ramification coverings comes from the Riemann existence theorem. The spaces that such coverings define go under the name of Hurwitz spaces and have been studied for a long time, starting from \cite{MR1511135, MR1509816}. We refer to \cite{MR260752} for a modern introduction.
Hurwitz spaces parametrizing odd ramification coverings have been studied by Fried et. al. \cite{MR2735035, MR1935406, Fried_connectednessof, MR667463}, where several relations with theta characteristics and modular towers are stated. 

Finding odd coverings where the source curve $C$  is \emph{fixed}  is much more involved and a Brill-Noether theory for coverings with special monodromy is not available. 
However, by using degeneration on Hurwitz spaces, the existence of coverings with degree $d \geq 2g-1$ and alternating monodromy group has been proved in \cite{MR2063042} for general curves in the sense of moduli. 
The degeneration is a powerful method and allowed in \cite{FMNP} to compute the number of odd ramification coverings, called the \emph{alternating Catalan numbers}, in the minimal degree case $d=2g-1$. Many other cases have been covered in \cite{Lian} with other techniques.

A completely different approach, also used in \cite{FMNP} for an alternative proof of the first step of the induction, appeared firstly in \cite{MR2093052}. It consists in the interpretation of an odd ramification covering $C \to \bP^1$ as a solution of a certain differential equation, see \cite[Section 4]{FMNP}.
A covering corresponds to a solution involving the meromorphic differential associated with a De-Rham problem. This technique is not so easy to carry on in general, but it becomes more accessible when symmetries are present. 

In this paper, we study a special class of odd coverings that we call hyperelliptic odd coverings.
Let $C \to \bP^1$ be a covering with $C$ a curve of genus $g \geq 1$ with an involution map $\sigma: C \to C$ such that the quotient is $\bP^1$. Denote by $P_1, \ldots, P_{2g+2}$ the fixed points of $\sigma$. When $g>1$ the curve $C$ is hyperelliptic and the points $P_i$ are the Weierstrass points of $C$. By an abuse of language, an elliptic curve $E=C$ is for us hyperelliptic. 
Consider also an involution $\iota: \bP^1 \to \bP^1$, and let $Z_\infty, Z_0$ be two fixed points of $\iota$.

\begin{definition}
The covering $H:C \to \bP^1$ is called \emph{hyperelliptic odd covering} if the following two conditions hold:
\begin{itemize}
    \item The map $H$ makes $\sigma$ and $\iota$ commute, i.e. $H(\sigma(P))= \iota (H(P)) \ \ \forall P\in C$.
    \item The fibre $H^{-1}(Z_\infty)$ is $\{P_1, \ldots, P_{2g+2}\}$ set-theoretically.
\end{itemize}
\end{definition}

We fix a model of the projective line $\bP^1=\bC \cup \infty$ and we assume $\iota$ is the involution induced by the multiplication by $-1$ on $\bC$, $Z_\infty=\infty$ and $Z_0=0$. In this setting, we can represent the map $H$ with a meromorphic function $h$ with poles only at the points $P_i$ and odd multiplicity. The function $h$ is defined up to a non-zero multiplicative constant. 
Since $g(C)\geq 1$, the map $H$ has at least three branch points.  A non-canonical way to fix $h$ is to fix a ramification point $R_1 \in C\setminus \{H^{-1}(0), H^{-1}(\infty)\}$ and assume $h(R_1)=1$.

Let $W=P_1 + \ldots + P_{2g+2}$ be the Weierstrass divisor of $C$. We choose an effective theta characteristic $L$ on $C$, and a section $s$ giving the isomorphism $L \cong \cO(F)$ where $F:=\sum_{P_i\in W}n_i P_i,$ $n_i\geq 0$ and $\deg F=\sum_{i=1}^{2g+2} n_i=g-1$. The differential equation governing the existence of hyperelliptic odd coverings $C \to \bP^1$ with theta characteristic $\cO(F)$ is the following.

\begin{equation} \label{main_equation}
\dd h =f^2\omega, \tag{$\star$}
\end{equation}

\noindent where $\omega=s^2 \in H^0(C, \omega_C)$ is a non-trivial holomorphic form which vanishes twice on $F$, i.e. $(\omega)=2F$. The choice of $F$ will play a very important role in the following.

Our first theorem is an existence result. Let $H^0(C,\sO_C(2F+W))$ be the space of meromorphic functions having poles at $2F+W$. Note that $\deg (2F+W)=4g.$ Let $V=H^0(C,\sO_C(2F+W))^{-}$ its anti-invariant part with respect to the involution $\sigma$ of $C$. 

\begin{theorem} \label{thm:Wsolutions}
The equation \eqref{main_equation} has solutions for suitable $f\in V$, $f\neq 0$. A non-zero solution $h \in V$ of \eqref{main_equation} defines a hyperelliptic odd covering.
\end{theorem}

We show that hyperelliptic odd coverings can be described in the projective space $\bP(V)$ by a locus $\Theta$ which is an intersection of $2g$ quadrics.
\medskip
 
Section \ref{sec:hurwsc} is devoted to studying the ramification data of these solutions from the point of view of monodromy. 
We first show in Proposition \ref{prop:minimaldegree} that when $C$ is general in the moduli space $H_g$ of hyperelliptic curves of genus $g$, the minimal degree of a hyperelliptic odd covering $H: C \to \bP^1$ is $4g$. Then, we give in Propositions \ref{prop:algebraicresult} and \ref{prop:topological} necessary and sufficient conditions on the monodromy data of a covering to be a hyperelliptic odd covering, i.e. to correspond to a solution of \eqref{main_equation}. Surprisingly, this implies that the choice of $F$ as before is equivalent to a certain monodromy data. 

This motivates the study of the Hurwitz space $\cH^{\text{HOC}}_g$ of hyperelliptic odd coverings of degree $4g$ modulo automorphisms of $\bP^1$. Let
$$\Phi: \cH^{\text{HOC}}_g \to H_g$$
be the forgetful map. Once we fix $C \in H_g$ general, $\Phi^{-1}(C)=\sqcup_F \cH_{C}(F)$, where $\cH_{C}(F)$ parametrizes the solutions of \eqref{main_equation} where $C$ and $F$ have been fixed. We prove in Corollary \ref{cor:finitelymany} a result which complements Theorem \ref{thm:Wsolutions} showing that \eqref{main_equation} has finitely many solutions.

In the same spirit of \cite{FMNP} we compute the number of such solutions. This is possible thanks again to Proposition \ref{prop:topological}, which guarantees that the forgetful map
$\cH^{\text{HOC}}_g \to H_g$ is dominant. If we fix a general $C$ in $H_g$, we can use deformation theory to show that the point of the fibre over $C$ are rigid, hence we can compute the number of elements of this fibre by using the fact that the points in $\Theta$ are smooth.

\begin{theorem} \label{thm:number}
The number of hyperelliptic odd ramification covering in $\cH_{C}(F)$ is $2^{2g}$ for every choice of $F$ on $C$. In total we have
$$ \deg \Phi = {3g \choose g-1} 2^{2g}.$$
\end{theorem}

In the last section, we consider the case of elliptic curves with an odd spin in degree four. This was already established in \cite{FMNP} with different techniques. We have to add that this was our starting case and that in this case, the $4$ solutions are almost explicit.

The approach of studying certain symmetries on differential equations seems very promising and could lead to a better understanding of special structures on Hurwitz spaces. Another important and natural problem to tackle will be to understand the case where $F$ is a non-effective spin, but it seems unlikely in this case to find symmetric solutions.

\subsection*{Plan of the paper.}
Section \ref{sec:oddhyperell} is devoted to the local study of \eqref{main_equation}, which leads to the proof of Theorem \ref{thm:Wsolutions}, and to the description of some properties of the solutions in \ref{cor:quadrica}. The viewpoint of monodromy and Hurwitz spaces is given in Section \ref{sec:hurwsc}. Finally, the number of solutions in the general case of degree $4g$ is computed in Section \ref{sec:numbrhyp}. In Section \ref{sec:elliptic} we consider the case of elliptic curves of degree four and odd spin.

\subsection*{Notation.}
Throughout the paper, we will work over the field $\bC$ of complex numbers. If we have a family $\sF$ parametrised by a certain scheme $V$, we say the \emph{general} element of $\sF$ satisfies a property $P$ if $P$ holds for every element in a Zariski dense open subset of $V$.

\subsection*{Acknowledgments}
The authors are members of GNSAGA (INDAM) and were supported by MIUR: Dipartimenti di Eccellenza Program (2018-2022)-Dept. of Math. Univ. of Pavia and by PRIN Project \emph{Moduli spaces and Lie theory} (2017).

\section{Hyperelliptic odd coverings with effective spin} \label{sec:oddhyperell}
In this section we will search hyperelliptic odd coverings with a fixed effective spin structure by showing that \eqref{main_equation} admits a solution. We will work in the same setting as in the introduction. Let $C$ be a hyperelliptic curve, $W=P_1 + \ldots P_{2g+2}$ be the divisor of Weierstrass points of $C$ and $\pi:C \to \bP^1$ be the hyperelliptic covering. We fix an effective theta characteristic by considering the divisor $$F:=\sum_{i=1}^{2g+2}n_i P_i, \quad n_i \geq 0$$
and assuming that $F$ has degree $g-1$. 

\begin{remark} \label{rem:numberchoices}
A choice for $F$ corresponds to distributing the degree $g-1$ among the $2g+2$ points of $W$. This behaves like constructing homogeneous polynomial of degree $g-1$ among $2g+2$ indeterminates, hence there are ${{3g} \choose g-1}$ choices for $F$. 
When $g>2$ the spin structure $\cO_C(F)$ can be even or odd. For instance, if $F=2P_1 + P_2 + \ldots + P_{g-2}$, then $h^0(C, \sO_C(F))$ is even, because $2P_1$ has two sections. 
In a similar way $F=P_1 + \ldots + P_{g-1}$ is odd. We refer to \cite{MR292836} and \cite{MR286136} for some more detail about theta characteristics and spin structures.
\end{remark}

Fix a form $\omega$ such that $(\omega)=2F$. We can find local coordinates $\{z_i, U_i\}$ centred at the point $P_i$ such that the form $\omega$ can be written in the following way: 
\begin{equation*} 
    \omega=z_i^{2n_i} \dd z_i,
\end{equation*}
and $z_i(\sigma(Q))=-z_i(Q)$ for every $Q \in U_i$.

The main result of this section is an existence result for the solutions of the certain differential equations of type \eqref{main_equation}.
Consider the divisor 
$$D:= 2F+W = \sum_{i=1}^{2g+2} (2n_i+1)P_i.$$
We have that:
$$\deg D = \sum_{i=1}^{2g+2} (2n_i+1)=2g+2+2\sum_{i=1}^{2g+2} n_i=2g+2+2(g-1)=4g.$$ 
By Riemann-Roch we have that $H^0(C,\cO_C(D))$, the space of meromorphic functions having poles at $D$, has dimension $3g+1$. The involution $\sigma$ acts on this space giving the decomposition
$$H^0(C,\cO_C(D))=H^0(C,\cO_C(D))^+ \oplus H^0(C,\cO_C(D))^-,$$
where $H^0(C,\cO_C(D))^+$ is the $\sigma$-invariant part and $H^0(C,\cO_C(D))^-$ is the anti-invariant part. Namely,
\begin{align*}
    H^0(C,\cO_C(D))^+&=\left\{f \in H^0(C,\cO_C(D)) \mbox{ such that } f(\sigma(z))=f(z)\right\},\\
    H^0(C,\cO_C(D))^-&=\left\{f \in H^0(C,\cO_C(D)) \mbox{ such that } f(\sigma(z))=-f(z)\right\}.
\end{align*}
\begin{lemma}
The space $H^0(C,\cO_C(D))^+$ has dimension $g$, and $H^0(C,\cO_C(D))^-$ has dimension $2g+1$.
\end{lemma}
\begin{proof}
Any $\sigma$-invariant function on $H^0(C,\cO_C(D))$ must come from $\bP^1$. Consider the divisor $F'=\pi_*(F)=\sum_{i=1}^{2g+2}n_i\pi(P_i)$ on $\bP^1$. The degree of $F'$ is also $g-1$. We have
$$H^0(C,\cO_C(D))^+=\pi^*H^0(\bP^1,\cO_{\bP^1}(F')).$$
This is enough to show that $H^0(C,\cO_C(D))^+$ has dimension $g$, and therefore, since the dimension of $H^0(C,\cO_C(D))$ is $3g+1$, we have that the space $H^0(C,\cO_C(D))^-$ has dimension $2g+1$.
\end{proof}

Let us denote $H^0(C,\cO_C(D))^-$ by $V$, and let $f \in V$ be a meromorphic function. We want to study the meromorphic form $f\omega$. Notice that by construction $f\omega$ has a pole of order at most $1$ in $P_i$. 
Let $a_i$ be the residue $\res_{P_i}(f\omega)$, for $i=1, \ldots, 2g+2$. Consider the following map
\begin{align*}
\gamma:V &\to \bC^{2g+2}\\   
    f &\mapsto (a_1, \ldots, a_{2g+2}).
\end{align*}

\begin{lemma} \label{lemma:descriptionW}
The map $\gamma$ is injective and has image $$L:=\left\{(x_1, \ldots, x_{2g+2}) : \sum_{i=1}^{2g+2} x_i = 0\right\}.$$ 
As a consequence, $V$ is isomorphic to $L$.
\end{lemma}
\begin{proof}
To show the injectivity: $\gamma(f)=0$ implies that $f \omega$ is an holomorphic $\sigma$-invariant form, hence $f=0$.
The image of $\gamma$ coincides with $L$ as a consequence of the global residues theorem.
\end{proof}

We remark that if $f \in V$, the form $f^2 \omega$ belongs to $H^0\left(C, \omega_C\left(2D\right)\right)$. Moreover, it has no residues, since it is anti-invariant, and the $\dd z_i/z_i$ are invariant under $\sigma$ (This can also be seen via an explicit computation, see the proof of Lemma \ref{lemma:quadric}). As a consequence, we can consider the map
\begin{align*}
\Psi:V &\to H^1(C, \bC)\\   
    f &\mapsto [f^2 \omega]_{\mbox{DR}}.
\end{align*}

\begin{proof}[Proof of Theorem \ref{thm:Wsolutions}]
Notice that if $\Psi(f)=0$ for a certain $f$, it means that there exists $h$ in $H^0(C, \cO_C(D))$ such that \eqref{main_equation} holds.
We can assume $h(0)=0$, and therefore that $h(-z)=-h(z)$. The zeroes of $\dd h$ are even, since it is an anti-invariant form. Therefore $h$ has odd ramification, with associated spin structure $\cO_C(F)$. The map $\Psi$ goes from $\bC^{2g+1} \to \bC^{2g}$, and the preimage $\Psi^{-1}(0)$ is at least $1$-dimensional. This shows the existence of non-trivial solutions.
\end{proof}

Recall that if $f \in \Psi^{-1}(0)$, then $\lambda f$ also belongs to $\Psi^{-1}(0)$, for $0 \neq \lambda \in \bC$. So we can consider the projectivization $\bP(V)=\bP^{2g}$. Let $\cO_{\bP(V)}(-1)$ be the tautological line bundle of $\bP(V)$. The map $\Psi$ induces a map
$$\psi: \cO_{\bP(V)}(-2) \to H^1(C,\bC).$$
Equivalently, $\psi=(\psi_1, \ldots, \psi_{2g})$ is a section of $\cO_{\bP(V)}(2)^{2g}$. Let $\cQ_i \subset \bP^{2g}$ be the quadric defined by the equation $\psi_i=0$. Set $\Theta=\bigcap_{i=1}^{2g} \cQ_i$.

\begin{remark} \label{rmk:quadrics}
Notice that $[h] \in \Theta$ if and only if $\psi(h)=0$, if and only if $h$ is a solution of \eqref{main_equation}. In the projective setting, the points of $h$ up to a non-zero constant give a hyperelliptic odd covering. There is a bijection between $\Theta$ and $\cH_C(F)$.
\end{remark}

We conclude this section to show that there is a canonical form for one of the quadrics $\cQ_i$. This will be only used in Section \ref{sec:elliptic}.

\begin{lemma} \label{lemma:quadric}
Assume $f \in V$ satisfies \eqref{main_equation}. Then the following equality holds
$$\sum_{i=1}^{2g+2}\frac{1}{2n_i+1}(a_{2n_i+1}^2) =0,$$
where $a_{2n_i+1} = \res_{P_i}(f \omega)$.
\end{lemma}
\begin{proof}
We are assuming there exists a $h \in V$ such that $dh=f^2\omega$. Consider as before coordinates $\{z_i,U_i\}$ in a neighbourhood of $P_i$. We have

\begin{align*}
    f(z_i)&=\sum_{k=0}^{n_i}\frac{a_{2n_i-2k+1}}{z^{2n_i-2k+1}}+\hol = \frac{a_{2n_i+1}}{z^{2n_i+1}}+\frac{a_{2n_i-1}}{z^{2n_i-1}}+\cdots+\frac{a_1}{z}+\hol\\
    h(z_i)&=\sum_{k=0}^{n_i}\frac{b_{2n_i-2k+1}}{z^{2n_i-2k+1}}+\hol = \frac{b_{2n_i+1}}{z^{2n_i+1}}+\frac{b_{2n_i-1}}{z^{2n_i-1}}+\cdots+\frac{b_1}{z}+\hol
\end{align*}

From this, we can explicitly compute in such neighbourhoods the expressions of $f^2\omega$ and $\dd h$.

\begin{align*}
    f^2(z_i)&=\sum_{k=0}^{n_i}\left(\sum_{t=0}^ka_{2n_i+1-2k+2t}a_{2n_i+1-2t}\right)\frac{1}{z^{4n_i+2-2k}}+\hol =\\
    &=\frac{a_{2n_i+1}^2}{z^{4n_i+2}}+\cdots+\hol\\
    f^2\omega(z_i)&=\left[\sum_{k=0}^{n_i}\left(\sum_{t=0}^k a_{2n_i+1-2k+2t}a_{2n_i+1-2t}\right)\frac{1}{z^{2n_i+2-2k}}+\hol\right]\dd z=\\& =(\frac{a_{2n_i+1}^2}{z^{2n_i+2}}+\cdots+\hol) \dd z\\
    \dd h(z_i)&=\left(\sum_{k=0}^{n_i}\frac{-(2n_i-2k+1)b_{2n_i-2k+1}}{z^{2n_i-2k+2}}+\hol\right)\dd z=\\
    & = (\frac{-(2n_i+1)b_{2n_i+1}}{z^{2n_i+2}}+\frac{-(2n_i-1)b_{2n_i-1}}{z^{2n_i}}+\cdots+\frac{-b_1}{z^2}+\hol) \dd z
\end{align*}

By using \eqref{main_equation} we can compare $f^2\omega$ and $\dd h$ for each degree and for each $i=1, \ldots, 2g+2$. However, the only relations which explicitly give a constrain on the space $V$ are the ones related to the term $k=0$, namely the coefficient of $z^{-2n_i+2}$. We get
$$(a_{2n_i+1})^2=-(2n_i+1)b_{2n_i+1},$$
up to multiplying by a nonzero constant. Recall that $a_{2n_i+1} = \res_{P_i}(f \omega)$, and $b_{2n_i+1} = \res_{P_i}(g \omega)$, namely the coordinates which describes the image $L$ of the map $\gamma$. Since $g \in V$ we have that $\sum_{i=1}^{2g} b_{2n_i+1} = 0$ by Lemma \ref{lemma:descriptionW}. As a consequence

$$\sum_{i=1}^{2g+2}\frac{1}{2n_i+1}(a_{2n_i+1}^2) =\sum_{i=1}^{2g+2}(- b_{2n_i+1})=-\sum_{i=1}^{2g} b_{2n_i+1} =0.$$
\end{proof}

Notice that this, considered as an equation in the coefficients $a_{2n_i+1}$, defines a quadric in $\bP(V)$.

\begin{corollary} \label{cor:quadrica}
Consider the space $L \subset\bC^{2g+2}$ defined in Lemma \ref{lemma:descriptionW}. The image of the solutions of \eqref{main_equation} are contained in the quadric cone
$$
\cQ:=\left\{(x_1, \dots, x_{2g+2}) : \sum_{i=1}^{2g+2}\frac{1}{2n_i+1}x_i^2=\sum_{i=1}^{2g+2}x_i=0\right\}
$$
which defines a smooth quadric in $\bP(L)$.
\end{corollary}
\begin{proof}
Just apply the map $\gamma$ to the result of Lemma \ref{lemma:quadric} and use Lemma \ref{lemma:descriptionW}. The smoothness can be checked directly.
\end{proof}

\section{The Hurwitz spaces viewpoint} \label{sec:hurwsc}
Let $\cH^{\text{HOC}}_g$ be the Hurwitz space parametrizing hyperelliptic odd coverings $C \to \bP^1$ modulo automorphisms of $\bP^1$, where $g$ is the genus of $C$. 
The aim of this section is to study $\cH^{\text{HOC}}_g$ from the point of view of monodromy. If $D=\sum n_i A_i$, with $A_i \neq A_j$ for $i \neq j$ is an effective divisor, we will denote its support by $\supp{D}$, that is $\supp{D}=\sum A_i$. Let $\bP^1=\bC \cup \{\infty\}$ and $\iota:\bP^1 \to \bP^1$ be the involution induced by the multiplication by $-1$ on $\bC$. Denote by $D_x$ the divisor induced by $H$ over the point $x \in \bP^1$. Recall that if $H$ is a hyperelliptic odd covering, the ramification divisor over the point $\infty$ is $D_\infty=\sum_{i=1}^{2g+2} (2n_i+1) P_i$, where $P_i$ are the Weierstrass points. In particular we have $\supp{D_\infty}=W$.

\begin{proposition} \label{prop:minimaldegree}
The minimal degree of a hyperelliptic odd covering $H:C \to \bP^1$ when $C$ is general in $H_g$ is $4g$. In this case, the branch divisor outside $\infty$ consists of $4g$ points interchanged by the action of $\iota$.
\end{proposition}
\begin{proof}
Let $d$ be the degree of $H$. By Riemann-Hurwitz we have
$$2g-2=-2d+\deg(D_\infty - \supp{D_\infty})+ \sum_{x \neq \infty }\deg(D_{x} - \supp{D_{x}}),$$
where we can sum just over the branch points.
By hypothesis we have $D_\infty - \supp{D_\infty}=d-(2g+2)$, so we have 
$$4g+d=\sum_{x \neq \infty }\deg(D_{x} - \supp{D_{x}}) \geq \sum_{x \neq 0, \infty }\deg(D_{x} - \supp{D_{x}}).$$

Notice that the maximum number of branch points outside $\infty$ is obtained when the ramification type is the minimum possible, that is $3$-cycles, namely the branch points different from $0, \infty$ will be $B_1, \dots, B_k$ such that $\iota B_i \neq B_j$ for all $i, j$, and $\iota B_1, \dots, \iota B_k$. The divisor $D_{B_i}$ is equal to $3R_i$ for a certain point $R_i$, $i=1, \dots, k$. With this assumption we have $\deg(D_{B_i} - \supp{D_{B_i}}) \geq 2$. We get 
$$4g+d \geq 2\sum_{i=1 }^{k}\deg(D_{B_i} - \supp{D_{B_i}}) \geq 4k.$$
Since the dimension of $H_g$ is $2g-1$ and the space of the automorphisms of $\bP^1$ fixing $0$ and $\infty$ has dimension $1$, we need at least $k=2g$, that is $4g$ points of branch of order at least $3$. It follows from the previous inequality that $d \geq 4g$, and $d=4g$ if and only if the point $0$ is not a branch and the branch locus outside $\infty$ is given by the $4g$ different points of order $3$.
\end{proof}

Now we want to study odd coverings in the minimal degree case $d=4g$ with the data prescribed by the previous proposition.
Fix $2g$ different points $\{B_1, \ldots, B_{2g}\}$ in $\bP^1$, such that $B_i \neq 0, \infty$ for every $i$. Denote by $B$ the set $\{\infty, B_1, \ldots, B_{2g}, \iota B_1, \ldots \iota B_{2g}\}$, which will become the branch locus. Consider the fundamental group $\pi_1(U, 0)$, where $U=\bP^1 \setminus B$. Let $D_\infty=\sum_{i=1}^{2g+2} (2n_i+1) P_i$, with $\sum n_i=g-1$. We will require that the monodromy data of the points $B_i$ is compatible with the involution $\iota$. 
We recall a standard result about the generators of the fundamental group $\pi_1(U, 0)$.

\begin{lemma} \label{lem:generatorsfund}
Let $\gamma_i$ be the class of a simple loop around the point $B_i$. The group $\pi_1(U, 0)$ is the free group generated by $\gamma_i$ and $\iota_* \gamma_i$, $i=1, \ldots, 2g$. Moreover we have
\begin{equation}\label{eqn:loopinfty}
    \gamma_\infty:=\gamma_1 \cdots \gamma_{2g} \iota_* \gamma_1 \cdots \iota_* \gamma_{2g},
\end{equation}
where $\gamma_\infty$ is the class of a simple loop around the point $\infty$.
\end{lemma}

Fix the involution $\ell:=(1,2)(3,4)\cdots(4g-1,4g)$ in $A_{4g}$ and consider a map $m:\pi_1(U, 0) \to A_{4g}$ with the following properties:

\begin{enumerate}[(i)]
\item $m(\gamma_i)$ is a $3$-cycle for $i=1, \dots, 2g$. \label{eqn:cond1}
\item \label{eqn:choicemon} $m(\iota_* \gamma_i)=\ell^{-1} m(\gamma_i) \ell$. 
\item \label{eqn:compmon}
$m(\gamma_\infty)$ is a product of disjoint cycles of length $(2n_i+1)$ with $\sum n_i=g-1$. 
\end{enumerate}

The choice of $\ell$ will be important later to ensure the compatibility with the involution $\iota$.
Since $\ell$ has order $2$, we have that $m(\gamma_\infty)=A\ell A\ell^{-1}=(A\ell)^2$, where $A\in A_{4g}$ is a product of the $3$-cycles $m(\gamma_1), \dots, m(\gamma_{2g})$. 

\begin{proposition} \label{prop:algebraicresult}
Fix an element $m(\gamma_\infty)$ as in \ref{eqn:compmon}. Then, there exists a map $m$ satisfying also conditions \ref{eqn:cond1}, \ref{eqn:choicemon}.
\end{proposition}
\begin{proof}
We want to show that there are $\tau_1, \dots, \tau_{2g}$ cycles of length $3$ such that  $m(\gamma_\infty)=(A\ell)^2$, where $A$ is the product of $\tau_i$. This is enough to obtain the map $m$ by using Riemann existence theorem. 

We first note that $m(\gamma_\infty)$ is a square in $A_{4g}$. Squares in the alternating group have been characterized in \cite[Lemma 4.3]{MR2183517}. Let $c_k$ be the number of cycles of length $k$ in the disjoint cycle decomposition of $m(\gamma_\infty)$. We need $c_{2k}$ even for all $k$, and this is true since $m(\gamma_\infty)$ has no cycles of even length by construction. For the same reason we have that $\sum c_{2k}$ is a multiple of $4$. This is enough to show that $m(\gamma_\infty)$ is a square in $A_{4g}$. As a consequence we have $m(\gamma_\infty)=B^2=(Ah)^2$ for a certain element $A \in A_{4g}$. 

Now notice that $A$ can be written as a product of $2g$ cycles of length $3$. This holds more generally: every element of $A_{n}$ can be written as a product of at most $\lfloor n/2 \rfloor$ cycles of length $3$. The case $n=3$ is trivially true. The induction step requires to analyse the two cases of a cycle of length $n$ if $n$ is odd, and of a product of two even cycles. Since the square of a cycle of length $3$ coincides with its inverse, we have also that every element of $A_{n}$ can be written as a product of exactly $\lfloor n/2 \rfloor$ cycles of length $3$.
Hence we have $A=\tau_1 \cdot \tau_2 \cdot \ldots \cdot \tau_{2g}$, and this concludes the proof.
\end{proof}

Now by the Riemann-Existence-Theorem, see for instance \cite{MR1326604}, we obtain a covering $C \to \bP^1$ from the map $m$. We prove now that it is a hyperelliptic odd covering.

\begin{proposition} \label{prop:topological}
Let $H:C \to \bP^1$ be any covering of degree $4g$ ramified over $\infty$ and over the $4g$ different points $B_1, \ldots, B_{2g}, \iota B_1, \ldots \iota B_{2g}$. Assume the monodromy data over the points $B_i$ satisfies \ref{eqn:cond1},\ref{eqn:choicemon} and \ref{eqn:compmon}. Then $H$ is a hyperelliptic odd covering.
\end{proposition}
\begin{proof}
We first recall the construction of $H$ starting from the monodromy data. Then we will define an involution on $C$ compatible with $\iota$, showing that the fibre over $\infty$ is composed of points fixed by such involution. 

Consider the following subgroup of $\pi_1(U, 0)$:
$$K_i:=\{\gamma \in \pi_1(U, 0) \st m(\gamma)(i)=i\}.$$
Let $M=\im(m)$ be the monodromy group of $H$. It is isomorphic to $\pi_1(U,0) / \cap_{i=1}^{4g} K_i$, and we can define $S_i:=m(K_i)=\{\sigma \in M \st \sigma(i)=i\}$. Now consider the covering $\tilde U$ associated with $K_1$. We have $\tilde U = \Delta/{K_1} \xrightarrow{\tilde H} U$, where $\Delta$ is the universal covering given by a complex unit disk. We denote the fibre $H^{-1}(0)$ by $\{Q_1, \ldots, Q_{4g}\}$, together with an isomorphism to $\{1, \ldots, 4g\}$.

Define a map $\bar \alpha: M \to M$ as $\bar \alpha(g):=\ell^{-1}g\ell$. The map $\bar \alpha$ sends $S_1$ to $S_2$. This follows immediately by the fact that $\bar \alpha$ is the conjugation with $\ell$. So at the level of monodromy, we have the following commutative diagram
$$
\xymatrix{
\pi_1(U,0) \ar[r]^{\iota_*}\ar[d]^m & \pi_1(U,0)\ar[d]^m\\
M \ar[r]^{\bar \alpha} & M.
}
$$

The map $\bar \alpha$ induces a map $\tilde \alpha: \tilde U \to \tilde U$, so that the following diagram commutes
$$
\xymatrix{
\tilde U\ar[r]^{\tilde \alpha}\ar[d]^{\tilde H} & \tilde U\ar[d]^{\tilde H}\\
U\ar[r]^\iota & U.
}
$$

Notice that $\ell$ has no fixed points, as a consequence the points over $0$ are not fixed by $\tilde \alpha$.

Let $H: C \to \bP^1$ be the covering obtained from $\tilde H$ where $C$ is a complete complex Riemann surface. The map $\tilde \alpha$ induces an involution $\alpha$ on $C$.
We need to show that $C$ is a hyperelliptic odd covering. First, thanks to \ref{eqn:cond1}, \ref{eqn:compmon} we know that $C$ is odd, and reasoning as in the proof of Proposition \ref{prop:minimaldegree} using the Riemann-Hurwitz formula, we have that $C$ has genus $g$. 

We will now show that the points $P_i \in H^{-1}(\infty)$ are actually fixed by $\alpha$. Since the fibre over $\infty$ is set-theoretically made by the points $P_i$, we need to show $\alpha(P_i)=P_i$ for every $P_i \in H^{-1}(\infty)$. This will prove that $\alpha$ is a hyperelliptic involution, since the $P_i$ are $2g+2$.

The action induced by $\alpha$ on $C$ is compatible with the action induced by $\iota$ on $\bP^1$. Consider the quotient of these actions $H': C' \to \bP^1$. Notice that the quotient of $\bP^1$ along the action induced by $\iota$ is still a $\bP^1$. We summarize these maps in the following commutative diagram, where $\eta$ and $\epsilon$ are the two quotient maps.

$$
\xymatrix{
C\ar[d]^{H}\ar[r]^{\eta} & C'\ar[d]^{H'}\\
\bP^1\ar[r]^{\epsilon} & \bP^1.
}
$$

Now consider the Riemann-Hurwitz formula applied to the covering $H': C' \to \bP^1$. We denote the genus of $C'$ by $g'$, and by $D_\infty$ the ramification divisor over the point $\infty$. Notice that $H'$ has still degree $4g$, and the ramification points are $\epsilon(B_i)=\epsilon(\iota B_i)$, $i=1, \ldots, 2g$. Moreover, the ramification over $\epsilon(B_i)$ is still a $3$-cycle. In total we get
$$2g'-2=-8g+4g+\deg(D_\infty - \supp{D_\infty}).$$
We want to show that $g'=0$, this will prove that $\alpha$ is a hyperelliptic involution. We have the inequality
\begin{equation} \label{eqn:inequality}
    \deg(D_\infty - \supp{D_\infty})=2g'+4g-2 \geq 4g-2.
\end{equation}
Now consider the covering $H'\circ \eta: C \to \bP^1$. This time the degree is $8g$, since $\eta$ is $2:1$. The ramification is as follows:
\begin{itemize}
\item At the point $0$ with contribution $4g$ to the Riemann-Hurwitz formula.
\item At the point $\infty$ with contribution $\deg(D_\infty - \supp{D_\infty})$ to the Riemann-Hurwitz formula.
\item At the points $B_i$, $i=1, \ldots, 2g$, each with contribution $4$ to the Riemann-Hurwitz formula, $2$ for the $3$-cycle and $2$ for the map $\eta$.
\end{itemize}
In total we get
$$2g-2=-16g + 4\cdot 2g + 4g + \deg(D_\infty - \supp{D_\infty})$$
Now assume that $k$ points among the $2g-2$ over $\infty$ are fixed. This allows us to explicitly write
\begin{equation}
    \deg(D_\infty - \supp{D_\infty})=\sum_{i=1}^k{4n_i+1}+\sum_{i=k+1}^{2g+2}{2n_i}
\end{equation}
Now remember that $\sum_{i=1}^{2g+2} n_i=g-1$. The only possibility for the value of $k$ in order to fulfill  (\ref{eqn:inequality}) is $2g+2$. This proves that $\alpha$ must fix all the point in the fibre over $\infty$, and consequently that the genus of $C'$ is zero.
\end{proof}

\begin{remark}
Proposition \ref{prop:topological} and Theorem \ref{thm:Wsolutions} give existence results for hyperelliptic odd coverings from two different perspectives.
The fact that a theta characteristic and the relative solution of \eqref{main_equation} correspond to a monodromy data is a well-known fact observed first by Serre in \cite{MR1078120} and by Fried in \cite{MR2735035}. In our case, the monodromy data also fix a section of the theta characteristic composed of Weierstrass points.
\end{remark}

\section{The number of hyperelliptic odd coverings} \label{sec:numbrhyp}
In this section, we compute the number of hyperelliptic odd covering $C \to \bP^1$ of degree $4g$ when the curve $C$ is a general element of $H_g$. 
We will call $\cH^{\text{HOC}}_g$ the Hurwitz space of hyperelliptic odd covering of genus $g$ and degree $4g$ modulo automorphisms of $\bP^1$ and we consider the forgetful map
$$\Phi:\cH^{\text{HOC}}_g \to H_g.$$
First at all from the results of the previous sections we have the following:

\begin{corollary} \label{cor:finitelymany}
The map $\Phi$ is generically finite, in particular, its differential $\dd \Phi$ is generically injective. 
\end{corollary}
\begin{proof}
Fix $C$ general. Theorem \ref{thm:Wsolutions} gives the existence of hyperelliptic odd coverings of degree lesser than or equal to $4g$, and Proposition \ref{prop:minimaldegree} shows that this degree must be $4g$. Then the general fibre of $\Phi$ is non-empty, hence $\Phi$ is dominant. By the same reasoning as in the proof of Proposition \ref{prop:minimaldegree}, we have that the dimension on the moduli is $2g-1$, the same as the dimension of $H_g$. This show that once we fix $C$ general in $H_g$, the number of hyperelliptic covering $C \to \bP^1$ which gives a solution of \eqref{main_equation} is finite. By the generic smoothness theorem, see \cite[Corollary 10.7]{MR0463157}, $\dd \Phi$ is generically injective. This concludes the proof.
\end{proof}

\subsection{Deformations of $H$ in the Hurwitz space.}
Consider the infinitesimal deformations of the map $H$ defined via the following short exact sequence, see \cite{MR2583634} and \cite{MR2247603}.

$$\xymatrix{
0\ar[r] & \cT_C\ar[r] & H^* \cT_{\bP^1}\ar[r] & \cN_H \ar[r] & 0.
}$$

We need to study the elements of $H^0(\cN_H)$ preserving the structure of hyperelliptic odd coverings.
Recall that we fixed the involution $\iota:\bP^1 \to \bP^1$ to be the multiplication by $-1$ on $\bC$, and that $H:C \to \bP^1$ is a hyperelliptic odd covering of degree $4g$ with $C$ general in $H_g$. 
Let the ramification points of $H$ be $P_1, \dots, P_{2g+2}$ over $\infty$ and $R_1, \dots, R_{2g}$, $\sigma R_1, \dots, \sigma R_{2g}$ the other $4g$ points. The following lemma is well known.

\begin{lemma} \label{lem:orderram}
Let $\bT \subset H^0(\cN_H)$ be the space of infinitesimal deformations which preserve the order of the ramification points of $H$. We have an isomorphism
$$\bT \cong \bigoplus_{S \text{ ramification}} \frac{\sO_C(-(r_S-1)S)}{\sO_C(-(r_S-2)S)},$$
where the sum ranges over all the ramification points $S$ of $H$, and $r_S$ is the order of $S$.
\end{lemma}
\begin{proof}
Consider a ramification point $S$ of order $r_S=r$. The map $H$ can be locally written around $S$ as $z \mapsto z^r=t$. In such coordinates the differential of $H$, which is a map $T_C \to H^* \cT_{\bP^1}$, can be written as $\partial / \partial z \mapsto (r-1)z^{r-1} \partial / \partial t$. We want to deform $S$ while keeping the order of ramification fixed. If we trivialize we get that the infinitesimal deformation correspond to the image if the following map:
$$\frac{\sO_C(-(r-1)S)}{\sO_C(-(r-2)S)} \hookrightarrow \frac{\sO_C(-(r-1)S)}{\sO_C} \cong \cN_H|_S.$$
Explicitly,  all the deformations are of the form $z^r+\epsilon \eta(z)$, with $\epsilon^2=0$. If we want to preserve the order of ramification, we have to consider $\eta(z)=\alpha rz^{r-1}$, so that $z^r+\epsilon \eta(z)=(z+\alpha \epsilon)^r$.
\end{proof}

There is a natural involution on $H^0(\cN_H)$ which we will call $\iota_*$. The infinitesimal variations which preserves the structure of hyperelliptic odd covering must be $\iota_*$-invariant. In order to describe this space, let $\cV$ be the space of deformation over the points $R_1, \dots, R_{2g}$ and define
\begin{equation} \label{eqn:somma}
\cV:=\bigoplus_{i=1}^{2g} \frac{\sO_C(-R_i)}{\sO_C(-2R_i)}.
\end{equation}

\begin{lemma} \label{lem:decomposeT}
We can decompose $\bT$ according to the ramification locus of $H$: 
$$\bT \cong \cV \oplus \iota_* \cV \oplus \bT_\infty,$$
where $\bT_\infty$ contains the infinitesimal deformations of the Weierstrass points $P_1, \dots, P_{2g+2}$.
The deformations on $\bT$ which are $\iota_*$-invariant, $\bT^{(\iota)}$, can be described with respect to this decomposition as the diagonal in $\cV \oplus \iota_* \cV$. More precisely the map $v$ defined as
\begin{align*}
v: \cV &\to \bT\\
\alpha &\mapsto (\alpha, \iota_* \alpha ,0),
\end{align*}
has image $\bT^{(\iota)}$.
\end{lemma}
\begin{proof}
Clearly the image of $v$ is $\bT^{(\iota)} \cap (\cV \oplus \iota_* \cV)$. Next we show that the component $\bT_\infty$ is anti-invariant. For a Weierstrass point $P_i$ with ramification of order $2k+1$, locally the deformations are $z^{2k} \partial / \partial t$. Around $P_i$, the involution $\iota$ lifts as the multiplication by $-1$ on the coordinate $z$, and $\iota_* (\partial / \partial t) = - \partial / \partial t$. As a consequence:
$$\iota_*(z^{2k}\partial/ \partial t)=(-z)^{2k}(-1)\partial / \partial t=-z^{2k}\partial / \partial t.$$
\end{proof}

Finally, we quotient out by the automorphisms of $\bP^1$. This is done by considering the following short exact sequence:
\[\xymatrix{
0\ar[r] & H^0(H^* \cT_{\bP^1})\ar[r] & H^0(\cN_H) \ar[r]^-{p} & \cQ \ar[r] & 0.
}\]
To avoid the quotient $\cQ$, we can represent the elements of $\bT^{(\iota)}$ in a non-canonical way by fixing $H(R_1)=1$, $H(\sigma R_1)=-1$. So we can consider the sum in \eqref{eqn:somma}  starting from $i=2$:
$$\overline \cV:=\bigoplus_{i=2}^{2g} \frac{\sO_C(-R_i)}{\sO_C(-2R_i)}.$$
We have $\overline \cV \hookrightarrow H^0(\cN_H)$ by setting zero also the deformations over $R_1$ and $\sigma R_1$. In this way we have an isomorphism $p (\overline \cV) \cong \cQ$.

\begin{proposition} \label{prop:defeqn}
Let $H: C \to \bP^1$ be a hyperelliptic odd covering with $C$ general in $H_g$. Every infinitesimal deformation of $H$ which keeps $C$ constant is zero (up to infinitesimal deformations of $\bP^1$).
\end{proposition}
\begin{proof}
We have a map $\overline \cV \to H^1(\cT_C)$. The image of this map is $\sigma$-invariant, so we get a morphism $\overline \cV \to \cT_{H_g}$ which by construction of $\cV$ is the differential of the map $\Phi$.
From Corollary \ref{cor:finitelymany} the map $\dd \Phi$ is generically injective, so the kernel of $\overline \cV \to \cT_{H_g}$ is zero. This concludes the proof.
\end{proof}

\subsection{Deformation of $H$ as a solution of \eqref{main_equation}}

Let $f, h \in V$ such that $\dd h=f^2 \omega$ and giving $C$ as the covering curve over $\bP^1$. We want to consider a first order deformation of $h$ which is still a solution of \eqref{main_equation}, let $h_1$, $f_1$ in $V$ such that
\begin{equation} \label{eqn:edef}
\dd (h + \epsilon h_1) = (f+\epsilon f_1)^2 \omega \qquad \mod \epsilon^2.
\end{equation}
This gives us the following system
\begin{equation} \label{eqn:sistema}
    \begin{cases}
    \dd h=f^2 \omega\\
     \dd h_1 = 2 f f_1 \omega.
    \end{cases}
\end{equation}

A solution of this system can be associated with a deformation of the map $H: C \to \bP^1$.

\begin{lemma}
The deformation  \eqref{eqn:edef} gives an element which belongs to the image of the map $v: \cV \to \bT$ defined in Lemma \ref{lem:decomposeT}.
\end{lemma}
\begin{proof}
Notice first that this deformation does not affect the Weierstrass points: consider appropriate local coordinates around a certain $P_i$ with order of ramification $\alpha_i$ such that the deformation of $h$ can be written as
$$z^{-\alpha_i} + \epsilon h_1(z)=z^{-\alpha_i}\left(1+\epsilon z^{\alpha_i} h_1(z)\right).$$
We change the coordinates on $\bP^1$ in the following way:
$$\frac{1}{z^{-\alpha_i}\left(1+\epsilon z^{\alpha_i} h_1(z)\right)}=\frac{z^{\alpha_i}}{1+\epsilon z^{\alpha_i} h_1(z)} = z^{\alpha_i}(1-\epsilon z^{\alpha_i} h_1(z))=z^{\alpha_i} - \epsilon z^{2\alpha_i} h_1(z).$$
The polynomial which gives the deformation is $z^{2\alpha_i} h_1(z)$, but every monomial has degree greater than $\alpha_i$, hence $h_1=0$. This shows that the component in the space $\bT_\infty$ is $0$.

Now we want to consider the ramification points outside $\infty$, which in the minimal degree case are the $4g$ points $R_i$ or order $3$. Again, consider appropriate local coordinates around the point $R_i$ such that $h(z)=z^3$ and $f(z)=z$. In the same coordinates let $h_1(z)=a_1z+a_2z^2$ and $f_1=b_0+b_1z+b_2z^2$. Recall that outside the poles $\omega=\dd z$, hence we get
\begin{align*}
\dd h_1 &= 2 f f_1 \omega\\
(a_1+2a_2z)\dd z &= 2z(b_0+b_1z+b_2z^2)\dd z\\
a_1 + 2a_2z &= 2 b_0 z +2 b_1 z^2 + 2 b_2 z^3.
\end{align*}
From which we get that $a_1=0$, and so $h_1(z)=a_2z^2$ preserves the ramification of order $3$ in $R_i$. The fact that $h$ and $h_1$ are anti-invariant with respect to $\iota$ shows that the image in $\cV \oplus \iota_* \cV$ agrees with the description of the map $v$.
\end{proof}

\begin{proposition} \label{prop:lambdasol}
If $C$ is general in $H_g$, then $h_1=\lambda h$ and $f_1=\lambda/2 f$ for $\lambda \in \bC$.
\end{proposition}
\begin{proof}
Consider $\dd h_1 / \dd h$. Thanks to the previous lemma, we can apply Proposition \ref{prop:defeqn} to get that $h_1$ has a zero of order at least $3$ near the points $R_i$. Moreover $h$ has poles of order greater than $h_1$ over the points $P_i$. As a consequence $\dd h_1 / \dd h$ is holomorphic, hence $\dd h_1=\lambda \dd h$ for a certain $\lambda \in \bC$. From $\dd(h_1 - \lambda h)=0$ it follows that also $h_1-\lambda h$ is constant. But since $h_1-\lambda h$ is also anti-invariant it is equal to $0$, hence we have $h_1=\lambda h$. From $\dd h_1=2f f_1 \omega$ it follows that $f_1=\lambda/2 f$. 
\end{proof}

\begin{proposition} \label{prop:thetasmooth}
If $C$ is general, the scheme $\Theta$ composed by solutions of \eqref{main_equation} is smooth.
\end{proposition}
\begin{proof}
Recall that $\Theta$ has been defined in Section \ref{sec:oddhyperell} as the intersection of the $2g$ sections of the map $\psi: \cO_{\bP(V)}(-2) \to H^1(C, \bC)$. Let $C(\Theta)$ be  the cone associated with $\Theta$ in the affine space $V$.
The tangent space to $f$ at $C(\Theta)$ is described by the second equation of \eqref{eqn:sistema}. Proposition \ref{prop:lambdasol} shows that $f_1$ is equal to $f$ up to a constant, hence in $\bP(V)$ this tangent space is zero. As a consequence the corresponding point of $\Theta$ are smooth.
\end{proof}

\begin{remark}
We can interpret this result also in terms of the map $\Psi: V \to H^1(C,\bC)$ of Section \ref{sec:oddhyperell}. We have that $\ker \Psi$ is one dimensional, generated by $f$.
\end{remark}

We conclude this section by computing the number of hyperelliptic odd coverings of minimal degree when $C$ is general.

\begin{proof}[Proof of Theorem \ref{thm:number}]
We know by Remark \ref{rmk:quadrics} that all the points in $\cH_C(F)$ belongs to $\Theta=\bigcap_{i=1}^{2g} \cQ_i$. Since each point of $\Theta$ is smooth by Proposition \ref{prop:thetasmooth}, the quadrics must intersect transversally and hence $\Theta$ consists of $2^{2g}$ points. 
Now from Remark \ref{rem:numberchoices} we get
$$\deg \Phi={{3g} \choose g-1} 2^{2g}.$$
\end{proof}

\section{Odd spin elliptic curves in degree four} \label{sec:elliptic}
We want to apply the techniques developed in the previous sections to study the case of $H:E \to \bP^1$ of degree $4$ and odd spin, with $E$ being a complex elliptic curve. In this more simplified case, we will able to argue not only about the existence of solution of \eqref{main_equation} but to compute also their number directly from the structure of $\Theta$ described in Section \ref{sec:oddhyperell}. This has already been addressed with different techniques in \cite[Section 4.1]{FMNP}.

The divisor $F$ defined in Section \ref{sec:oddhyperell} is now $0$, so $D=P_1+P_2+P_3+P_4$, where the points $P_i$ are invariant with respect to a fixed involution on $E$. 
The curve $E$ can be seen as a quotient of $\bC$ by a suitable lattice, and the generator of $H^0(E,\omega_E)$ induced by this quotient will be denoted by $\dd z$. Equation \eqref{main_equation} becomes
$$\dd h=f^2 \dd z.$$

We can apply verbatim the previous approach. In particular, we can construct the space $V$ defined in Section \ref{sec:oddhyperell}, which in this case has dimension $3$. There is a isomorphism $\gamma: V \to L$ , where $L \subset \bC^{2g+2}$ given by $\sum_{i=1}^{2g+2}x_i=0$ has been defined in Lemma \ref{lemma:descriptionW}.  
There are two conics $\cQ_1$, $\cQ_2$ in $\bP(V) \cong \bP(L)$ such that $[f]\in \cQ_1 \cap \cQ_2$ if and only if $f$ is a solution of \eqref{main_equation}. 
We assume that $\cQ_1$ is given by the equation of Corollary \ref{cor:quadrica}, and in particular is smooth. We will show that $\cQ_1$ and $\cQ_2$ meet transversally, that is $\#\Theta=4$, where $\Theta:= \cQ_1 \cap \cQ_2$.

\begin{theorem}
Consider the elliptic curve $E$ as before. The locus $\Theta$ consists of $4$ distinct points interchanged by the action of the points of order $2$ of $E$.
\end{theorem}
\begin{proof} 
The translations by order $2$-points induce a natural action of $K:=\bZ/2\bZ\times \bZ/2\bZ$ on all the spaces defined before.
Indeed, if for instance $\tau:=P_2-P_1$, $\eta =P_3-P_1$ and $\zeta=P_4-P_1$, we have that the divisor $D$ is $\tau$-invariant:  
 $$\tau(P_1)=P_2; \ \tau(P_2)=P_1;\ \tau(P_3)=P_4; \ \tau(P_4)=P_3.$$
 The action on the residues of $f\in V$ is given by
$$(a,b,c,d)^\tau=(b,a,d,c).$$
The other points of order $2$, act similarly. 
It follows that the action of $K$ on $L$ is induced by the standard representation. Also the locus $\Theta$ and $\cQ_1$ are invariant under the action of $K$. 
There are six points of $\cQ_1$ which have a $K$-orbit of cardinality smaller than $4$, corresponding to the fixed points of the involutions induced by $\tau$, $\eta$ and $\zeta$ on $\cQ_1$: for instance we have $T_1:=[v_1]$ and $T_2:=[v_2]$, with $v_1:=(1,-1,i,-i)$ and $v_2:=(1,-1,-i,i)$ are fixed by $\tau$.

We now show that $\Theta=\psi^{-1}(0)$ is zero-dimensional. Note that $\Theta\subset \cQ_1$ where $\cQ_1$ is irreducible. Then either $\Theta$ has dimension $0$ or $\Theta=\cQ_1$. To see that $\Theta\neq \cQ_1$ it is enough to prove that either $T_1$ or $T_2$ are not in $\gamma(\Theta)$. Assume by contradiction this is not the case. Let $f_1,$ $f_2$ be solutions of \eqref{main_equation} such that
$\gamma(f_i)=v_i$ for $i=1,2$. 
We would have $dh=f_1^2dz=f_2^2dz$, where $\gamma(h)=(1,1,-1,-1)$. This implies either $f_1=f_2$ or $f_1=-f_2$. Since $v_1\neq v_2$ and $v_1\neq -v_2$ we obtain a contradiction. 
We have proven so far that $\Theta$ is zero-dimensional. Therefore the cardinality of $\Theta$ is smaller than or equal to $4$. 

 Now we will prove that neither $T_1$ nor $T_2$ are in $\gamma(\Theta)$. 
We have that $\eta(T_1)=T_2$. It follows that $T_1\in \gamma(\Theta)$ if and only if $T_2\in \gamma(\Theta)$, this is not possible. 
The same holds for the fixed points of $\eta$ and $\zeta$. It follows that $\Theta$ must be a full orbit of $K$ of cardinality $4$.
\end{proof}

\begin{remark}
The fact that of $\Theta$ is finite can be seen as well from the theory of Hurwitz spaces. It has been proved by Fried in \cite{MR2735035} that there are only two connected components of $\cH=\cH^{\text{odd}}\sqcup \cH^{\text{even}},$ corresponding to the even and the odd spin structures. The coverings constructed in the previous section belong to $\cH^{\text{odd}}$ and are monodromy-invariant. It follows from the irreducibility that all the coverings of $\cH^{\text{odd}}$ can be obtained in this way. We remark that this can be generalised to higher degrees and different ramification type by using an approach similar to \cite{MR3680993} to compute the number of irreducible components of the Hurwitz spaces. It would be interesting to generalise this more explicit computation to hyperelliptic curves, and counting the number of connected components of $\cH^{\text{HOC}}_g$.
\end{remark}

\bibliographystyle{siam}
\bibliography{bibliografia}

\end{document}